\documentclass[a4paper]{amsart}
\usepackage{textcomp, amsmath, amssymb, amsthm}
\usepackage{graphicx}
\usepackage{enumerate}
\usepackage[margin=1in]{geometry}

\usepackage{physics}
\usepackage[colorlinks,
linkcolor=black,
anchorcolor=black,
citecolor=black,
urlcolor=black
]{hyperref}

\usepackage{tikz} 

\tikzset{every picture/.style={line width=0.75pt}} 

\newtheorem*{theorem*}{Theorem}
\newtheorem{theorem}{Theorem}[section]
\newtheorem{lemma}[theorem]{Lemma}

\newtheorem{proposition}[theorem]{Proposition}

\theoremstyle{definition} 

\newtheorem*{defin*}{Definition}
\newtheorem{defin}[theorem]{Definition}

\title{Topological conjugacy between Schneider's continued fraction map and a shift map on $\mathbb{Q}_p$}
\author{HANWEN LIU}
\date{November 17, 2023}

\begin{document}

\maketitle

\begin{abstract}
    We prove that Schneider's continued fraction map is topologically conjugate to a shift map defined on $\mathbb{Q}_p$, and the topological conjugation $f\colon\mathbb{Q}_p \rightarrow \mathbb{Q}_p$ is an isometry such that $f(\mathbb{Q})=\mathbb{Z}[\frac{1}{p}]$. Furthermore, for $x \in \mathbb{Q}_p$ we proved that $x \in \mathbb{Q}$ if and only if $f(x) \in \mathbb{Z}[\frac{1}{p}]$.
\end{abstract}

\quad

\section{Introduction and Backgrounds}

\quad

Throughout this article, let $p \in \mathbb{N}$ be an arbitrary but fixed rational prime.

As usual, denote by $v:=v_p\colon \mathbb{Q}_p \rightarrow \mathbb{Z} \cup\{\infty\}$ the $p$-adic valuation and $| \cdot \vert_p$ the absolute value on $\mathbb{Q}_p$ defined by $|x\vert_p:=p^{-v(x)}$ for every nonzero $x \in \mathbb{Q}_p$.
\begin{defin}\label{definition_1_1}
    Define $\epsilon\colon \mathbb{Q}_p^{\times} \rightarrow \mathbb{Z}_p^{\times}$ by $\epsilon(x):=|x\vert_p x$, and by convention, let $\epsilon(0)=1$.
\end{defin}

    Recall that $0 \rightarrow \mathbb{Z} \rightarrow \mathbb{Q}_p^{\times} \stackrel{\epsilon}{\rightarrow} \mathbb{Z}_p^{\times} \rightarrow 0$ is an exact sequence of Abelian groups.

\begin{defin}\label{definition_1_2}
    For any integer sequence $\left(a_n\right)_{n=0}^{\infty}$ such that $0 \leqslant a_n \leqslant p-1$ for all $n \in \mathbb{N}$, define $\left\lfloor\sum_{n=0}^{\infty} a_n p^n\right\rfloor_p:=a_0$.
\end{defin}

Recall that for any $x \in \mathbb{Z}_p$, $\lfloor x\rfloor_p$ is the unique element of $\{0, \ldots, p-1\}$ such that $|x-\lfloor x\rfloor_p\vert_p<1$.
\begin{defin}\label{definition_1_3}
    Define Schneider's continued fraction map $\tau\colon \mathbb{Q}_p \rightarrow \mathbb{Q}_p$ by 
    $$\tau(x):=\dfrac{1}{\epsilon(x)}-\left\lfloor\dfrac{1}{\epsilon(x)}\right\rfloor_p$$
\end{defin}

The ergodic theory and entropy of dynamical system $\left(\mathbb{Q}_p, \tau\right)$ is studied in [1] and [2].
\begin{defin}\label{definition_1_4}
    Define $\sigma\colon \mathbb{Q}_p \rightarrow \mathbb{Q}_p$ by $\sigma(x):=\epsilon(x)-\lfloor\epsilon(x)\rfloor_p$.
\end{defin}

The following statement describes the dynamics of $\sigma$, which is essentially a shift on $\mathbb{Q}_p$.
\begin{proposition}\label{proposition_1_5}
    Let $\left(a_n\right)_{n=0}^{\infty}$ and $\left(b_n\right)_{n=0}^{\infty}$ be sequences of rational numbers such that
    $$
    \left\{a_n \mid n \in \mathbb{N}\right\} \subseteq\{1, \ldots, p-1\},\ \left\{b_n \mid n \in \mathbb{N}\right\} \subseteq\left\{p^e \mid e \in \mathbb{Z}\right\} \cup\{0\}
    $$
    and $\left|b_{n+1}\right|_p<1$ for all $n \in \mathbb{N}$, then $\sigma\left(\sum_{n=0}^{\infty} a_n \prod_{m=0}^n b_m\right)=\sum_{n=1}^{\infty} a_n \prod_{m=1}^n b_m$.
\end{proposition}
\begin{proof}
    The desired result follows immediately from the very definition of $\sigma$.
\end{proof}

Using the shift map $\sigma$, one can write the $p$-adic expansion of a number as a closed formula.

\begin{proposition}\label{proposition_1_6}
    Let $x \in \mathbb{Q}_p$, then $x=\sum_{n=0}^{\infty} \lfloor \epsilon\left(\sigma^n(x)\right)\rfloor_p\prod_{m=0}^n \sigma^m(x) / \epsilon\left(\sigma^m(x)\right)$.
\end{proposition}
\begin{proof}
    This is the $p$-adic expansion of $x$.
\end{proof}

Proposition \ref{proposition_1_6} motivates the following definition.
\begin{defin}\label{definition_1_7}
    Define $f\colon \mathbb{Q}_p \rightarrow \mathbb{Q}_p$ by $f(x):=\sum_{n=0}^{\infty}\left\lfloor 1 / \epsilon\left(\tau^n(x)\right) \right\rfloor_p \prod_{m=0}^n \tau^m(x) / \epsilon\left(\tau^m(x)\right)$.
\end{defin}

In section \ref{sec2}, we prove that $f$ is an isometry such that the functional equation $f \circ \tau=\sigma \circ f$ holds.

In section \ref{sec3}, we prove that $f(\mathbb{Q})=\mathbb{Z}[\frac{1}{p}]$, and moreover for arbitrary $x \in \mathbb{Q}_p$ we prove that $x \in \mathbb{Q}$ if and only if $f(x) \in \mathbb{Z}[\frac{1}{p}]$.

\newpage 
\section{Proof of the Main Result}\label{sec2}

\quad

We use Pringsheim's notation for continued fractions.

We shall first present several basic properties of $p$-adic continued fractions that are parallel with those of the classical continued fractions in $\mathbb{R}$.

\begin{proposition}\label{proposition_2_1}
    For any $a_0, \ldots, a_n \in \mathbb{Z}_p^{\times}$ and $b_0, \ldots, b_n \in \mathbb{Q}_p$, if $b_1, \ldots, b_n \in p \mathbb{Z}_p$, then the formula 
    $$
    \left|\dfrac{b_0}{a_0 +} \cdots \dfrac{b_{n-1}}{a_{n-1}+} \dfrac{b_n}{a_n+x}-\dfrac{b_0}{a_0 +} \cdots \dfrac{b_{n-1}}{a_{n-1} +} \dfrac{b_n}{a_n+y}\right|_p=|x-y\vert_p \prod_{i=0}^n\left|b_i\right|_p \leqslant p^{-n}\left|b_0\right|_p
    $$
    holds for all $x, y \in p \mathbb{Z}_p$.
\end{proposition}
\begin{proof}
    Since $\left|\dfrac{b_n}{x+a_n}-\dfrac{b_n}{y+a_n}\right|_p=\dfrac{\left|b_n(y-x)\right|_p}{\left|x+a_n\right|_p \cdot\left|y+a_n\right|_p}=\left|b_n\right|_p \cdot|x-y|_p$, by induction
    $$
    \begin{aligned}
        \left|\dfrac{b_0}{a_0 +} \cdots \dfrac{b_{n-1}}{a_{n-1} +} \dfrac{b_n}{a_n+x}-\dfrac{b_0}{a_0 +} \cdots \dfrac{b_{n-1}}{a_{n-1} +} \dfrac{b_n}{a_n+y}\right|_p=&\left|\dfrac{b_n}{x+a_n}-\dfrac{b_n}{y+a_n}\right|_p \cdot \prod_{i=0}^{n-1}\left|b_i\right|_p\\ 
        =&|x-y|_p \cdot \prod_{i=0}^n\left|b_i\right|_p \leqslant p^{-n}\left|b_0\right|_p
    \end{aligned}
    $$
    as desired.
\end{proof}
\begin{proposition}\label{proposition_2_2}
    Let $\left(a_n\right)_{n=0}^{\infty}$ be a sequence in $\mathbb{Z}_p^{\times}$ and let $\left(b_n\right)_{n=0}^{\infty}$ be a sequence in $\mathbb{Q}_p$, assume that $\left|b_{n+1}\right|_p<1$ for all $n \in \mathbb{N}$, then there exists a unique $x \in \mathbb{Q}_p$ such that 
    $$\dfrac{b_0}{a_0+} \cdots \dfrac{b_{n-1}}{a_{n-1}+} \dfrac{b_n}{a_n} \longrightarrow x$$
    as $n \rightarrow \infty$ in $p$-adic norm.
\end{proposition}
\begin{proof}
    For any $n \in \mathbb{N}$ denote 
    $$P_n:=\dfrac{b_0}{a_0 +} \cdots \dfrac{b_{n-1}}{a_{n-1} +} \dfrac{b_n}{a_n}$$
    By Proposition $2.1$, $\left|P_{n+1}-P_n\right|_p \leqslant p^{-n}\left|b_0\right|_p \rightarrow 0$ as $n \rightarrow \infty$ in $p$-adic norm. Since $\mathbb{Q}_p$ is complete non-Archimedean and $P_n=P_0+\sum_{i=0}^{n-1}\left(P_{i+1}-P_i\right)$ for all $n \in \mathbb{N}$, we have that the sequence $\left(P_n\right)_{n=0}^{\infty}$ converges in $p$-adic norm.
\end{proof}
\begin{proposition}\label{proposition_2_3}
    Let $x \in \mathbb{Q}_p$, and for each $n \in \mathbb{N}$ let 
    $$
    x_n:=\tau^n(x),\ a_n:=\left\lfloor1 / \epsilon\left(x_n\right)\right\rfloor_p,\ b_n:=x_n / \epsilon\left(x_n\right)
    $$
    then 
    $$x=\dfrac{b_0}{a_0 +} \cdots \dfrac{b_{n-1}}{a_{n-1} +} \dfrac{b_n}{a_n+x_{n+1}}$$
    for all $n \in \mathbb{N}$, and 
    $$\dfrac{b_0}{a_0 +} \cdots \dfrac{b_{n-1}}{a_{n-1}+} \dfrac{b_n}{a_n} \rightarrow x$$
    as $n \rightarrow \infty$ in $p$-adic norm.
\end{proposition}
\begin{proof}
    By the definition of $\tau$, the equality 
    $$x_n=\dfrac{b_n}{a_n+\tau\left(x_n\right)}$$
    holds for all $n \in \mathbb{N}$. Induction on $n \in \mathbb{N}$ then yields that 
    $$x=\dfrac{b_0}{a_0+} \cdots \dfrac{b_{n-1}}{a_{n-1}+} \dfrac{b_n}{a_n+x_{n+1}}$$
    for all $n \in \mathbb{N}$. For each $n \in \mathbb{N}$ denote 
    $$P_n:=\dfrac{b_0}{a_0 +} \cdots \dfrac{b_{n-1}}{a_{n-1} +} \dfrac{b_n}{a_n}$$
    then $\left(P_n\right)_{n=0}^{\infty}$ converges in $p$-adic norm by Proposition \ref{proposition_2_2}. Since by Proposition \ref{proposition_2_1} $\left|x-P_n\right|_p<p^{-n}|x|_p$ for all $n \in \mathbb{N}$, we obtain that $P_n \rightarrow x$ as $n \rightarrow \infty$ in $p$-adic norm.

\end{proof}

\clearpage

\begin{lemma}\label{lemma_2_4}
    Let $e \in \mathbb{Z}, a \in\{1, \ldots, p-1\}$ and $x \in p \mathbb{Z}_p$, then $f\left(\dfrac{p^e}{x+a}\right)=p^e(f(x)+a)$.
\end{lemma}
\begin{proof}
    Denote $y:=\dfrac{p^e}{x+a}$ then we have that
    $$
    e=v_{p}(y),\ a=\lfloor1 / \epsilon(y)\rfloor_p,\ x=\tau(y)
    $$
    and hence $f(y)=p^e\left(a+\sum_{n=1}^{\infty}\left\lfloor 1 / \in\left(\tau^{n-1}(x)\right)\right\rfloor_p \prod_{m=1}^n \tau^{m-1}(x) / \in\left(\tau^{m-1}(x)\right)\right)=p^e(f(x)+a)$.
\end{proof}
\begin{lemma}\label{lemma_2_5}
    Let $a_0, \ldots, a_n \in\{1, \ldots, p-1\}$ and $e_0, \ldots, e_n \in \mathbb{Z}$, if $e_1, \ldots, e_n>0$, then the formula 
    $$
    f\left(\dfrac{p^{e_0}}{a_0 +} \cdots \dfrac{p^{e_{n-1}}}{a_{n-1} +} \dfrac{p^{e_n}}{a_n+x}\right)=f(x) p^{e_0+\cdots+e_n}+\sum_{i=0}^n a_i p^{e_0+\cdots+e_i}
    $$
    holds for all $x \in p \mathbb{Z}_p$.
\end{lemma}
\begin{proof}
    Apply Lemma \ref{lemma_2_4}.
\end{proof}

We are now prepared to state the main results of this section.

\begin{theorem}\label{theorem_2_6}
    It holds that $f \circ \tau=\sigma \circ f$.
\end{theorem}
\begin{proof}
    For arbitrary $x \in \mathbb{Q}_p$, by Proposition 1.5 $\sigma(f(x))=\sum_{n=1}^{\infty}\lfloor 1 / \epsilon\left(\tau^n(x)\right) \rfloor_p \prod_{m=1}^n \tau^m(x) / \epsilon\left(\tau^m(x)\right)$ $=\sum_{n=0}^{\infty}\lfloor 1/ \epsilon\left(\tau^{n+1}(x)\right)\rfloor_p \prod_{m=0}^n \tau^{m+1}(x) / \in\left(\tau^{m+1}(x)\right)=f(\tau(x))$.
\end{proof}
\begin{lemma}\label{lemma_2_7}
    The $p$-adic function $f$ is an isometric embedding.
\end{lemma}
\begin{proof}
    Take any distinct $x, y \in \mathbb{Q}_p$. For each $n \in \mathbb{N}$ denote
    $$
    x_n:=\tau^n(x),\ a_n:=\left\lfloor 1 / \epsilon\left(x_n\right)\right\rfloor_p,\ b_n:=x_n / \epsilon\left(x_n\right)
    $$
    and $y_n:=\tau^n(y)$.

    Provided that $\lfloor 1 / \epsilon\left(y_n\right)\rfloor_p=a_n$ and $y_n / \epsilon\left(y_n\right)=b_n$ for all $n \in \mathbb{N}$, then by Proposition \ref{proposition_2_3} $x=y$, contradiction. Therefore 
    $$m:=\sup \left\{n \in \mathbb{N} \mid\left(a_i, b_i\right)=\left(\lfloor 1 / \epsilon\left(y_i\right)\rfloor_p, y_i / \epsilon\left(y_i\right)\right) \forall i \in\{0, \cdots, n\}\right\}$$
    is a finite number.
    By Proposition \ref{proposition_2_1} and Proposition \ref{proposition_2_3}, we obtain that
    $$
    \begin{aligned}
    |f(x)-f(y)|_p & =\exp \left(-\ln (p) \sum_{i=0}^{m+1} \min \left\{v\left(x_n\right), v\left(y_n\right)\right\}\right) \\
    & =\max \left\{\left|x_{m+1}\right|_p,\left|y_{m+1}\right|_p\right\} \prod_{i=0}^{m-1}\left|x_i\right|_p\\ 
    & =\left|x_{m+1}-y_{m+1}\right|_p \cdot \prod_{i=0}^{m-1}\left|b_i\right|_p\\ 
    & =\left|\dfrac{b_0}{a_0 +} \cdots \dfrac{b_{m-1}}{a_{m-1}+}+\dfrac{b_m}{a_m+x_{m+1}}-\dfrac{b_0}{a_0 +} \cdots \dfrac{b_{m-1}}{a_{m-1}+} \dfrac{b_m}{a_m+y_{m+1}}\right|_p\\ 
    & =|x-y|_p.
    \end{aligned}
    $$
    This concludes the proof.
\end{proof}
\begin{lemma}\label{lemma_2_8}
    The $p$-adic function $f$ is surjective.
\end{lemma}
\begin{proof}
    Take an arbitrary $y \in \mathbb{Q}_p$. For each $n \in \mathbb{N}$ denote $a_n:=\left\lfloor1 /\epsilon\left(\sigma^n(y)\right)\right\rfloor_p$ and $b_n:=\sigma^n(y) / \epsilon\left(\sigma^n(y)\right)$. By Proposition \ref{proposition_2_2} there exists $x \in \mathbb{Q}_p$ such that 
    $$\dfrac{b_0}{a_0 +} \cdots \dfrac{b_{n-1}}{a_{n-1}+} \dfrac{b_n}{a_n} \rightarrow x$$
    as $n \rightarrow \infty$ in $p$-adic norm. By Lemma \ref{lemma_2_7} in particular $f$ is continuous. Therefore by Lemma \ref{lemma_2_5} and Proposition \ref{proposition_1_6}, $f(x)=\sum_{n=0}^{\infty} a_n \prod_{m=0}^n b_m=y$.
\end{proof}
\begin{theorem}\label{theorem_2_9}
    The $p$-adic function $f$ is an isometry.
\end{theorem}
\begin{proof}
    Combine Lemma \ref{lemma_2_7} and Lemma \ref{lemma_2_8}.
\end{proof}

\newpage
\section{A Criterion for Rationality of \texorpdfstring{$p$-adic } -Numbers}
\label{sec3}

\quad

\begin{proposition}\label{proposition_3_1}
    The integers 0 and $-p$ are fixed points of $f$.
\end{proposition}
\begin{proof}
    By the very definition of $f$, we have $f(0)=0$. Since 
    $$-p=\sum_{n=1}^{\infty}(p-1) p^n=\dfrac{p}{p-1+} \cdots \dfrac{p}{p-1+} \dfrac{p}{p-1+\cdots}$$
    we obtain that $f(-p)=-p$.
\end{proof}
\begin{theorem}\label{theorem_3_2}
    Let $x \in \mathbb{Q}$, then there exists $n \in \mathbb{N}$ such that $\tau^n(x)=0$ or $\tau^n(x)=-p$.
\end{theorem}
\begin{proof}
    This is a direct consequence of the Satz in [3].
\end{proof}
\begin{proposition}\label{proposition_3_3}
    Let $x \in \mathbb{Q}$, then $f(x) \in \mathbb{Z}[\frac{1}{p}]$.
\end{proposition}
\begin{proof}
    By Proposition \ref{proposition_3_1} $f(0)=0$ is an integer.

    Fix an arbitrary $x \in \mathbb{Q}^{\times}$, and for each $n \in \mathbb{N}$ denote 
    $$
    x_n:=\tau^n(x),\ a_n:=\lfloor 1 / \epsilon\left(x_n\right) \rfloor_p,\ e_n:=v\left(x_n\right)
    $$

    By Theorem \ref{theorem_3_2} there exists natural number $m \in \mathbb{N}$ such that $x_{m+1}=0$ or $x_{m+1}=-p$, and $e_n<\infty$ for all $n \in\{0, \ldots, m\}$.

    Since by Proposition \ref{proposition_2_3} 
    $$x=\dfrac{p^{e_0}}{a_0 +} \cdots \dfrac{p^{e_{m-1}}}{a_{m-1}+} \dfrac{p^{e_m}}{a_m+x_{m+1}}$$ 
    and by Proposition \ref{proposition_3_1}. $f\left(x_{m+1}\right)=x_{m+1}$, Lemma \ref{lemma_2_5} yields that $$f(x)=p^{e_0+\cdots+e_m}x_{m+1}+\sum_{i=0}^m a_i p^{e_0+\cdots+e_i} \in \mathbb{Z}[1/p]$$ 
    as desired.
\end{proof}
\begin{proposition}\label{proposition_3_4}
    For any $y \in \mathbb{Z}[\frac{1}{p}]$, there exists $x \in \mathbb{Q}$ such that $y=f(x)$.
\end{proposition}
\begin{proof}
    If $y=0$ then $f(0)=0=y$ by Proposition \ref{proposition_3_1}. 
    
    If $y>0$, then there exists $a_0, \ldots, a_n \in\{1, \ldots, p-1\}$ and $e_0, \ldots, e_n \in \mathbb{Z}$ with $e_1, \ldots, e_n>0$, such that $\sum_{i=0}^n a_i p^{e_0+\cdots+e_i}$. Let 
    $$x':=\dfrac{p^{e_0}}{a_0+} \cdots \dfrac{p^{e_{n-1}}}{a_{n-1}+} \dfrac{p^{e_n}}{a_n} \in \mathbb{Q}$$
    then by Lemma \ref{lemma_2_5}, $y=f(x')$.

    If $y<0$, then there exists $a_0, \ldots, a_n \in\{1, \ldots, p-1\}$ and $e_0, \ldots, e_n \in \mathbb{Z}$ with $e_1, \ldots, e_n>0$, such that $y=\sum_{i=0}^n a_i p^{e_0+\cdots+e_i}-p^{e_0+\cdots+e_n+1}=f(-p) p^{e_0+\cdots+e_n}+\sum_{i=0}^n a_i p^{e_0+\cdots+e_i}$. Let 
    $$x'':=\dfrac{p^{e_0}}{a_0 +} \cdots \dfrac{p^{e_{n-1}}}{a_{n-1}+}+\dfrac{p^{e_n}}{a_n-p} \in \mathbb{Q}$$
    then by Lemma \ref{lemma_2_5}, $y=f(x'')$.
\end{proof}
\begin{theorem}\label{theorem_3_5}
    Let $x \in \mathbb{Q}_p$ be a p-adic number, then $x \in \mathbb{Q}$ if and only if $f(x) \in \mathbb{Z}[\frac{1}{p}]$. Furthermore, it holds that $f(\mathbb{Q})=\mathbb{Z}[\frac{1}{p}]$.
\end{theorem}
\begin{proof}
    Notice that by Theorem \ref{theorem_2_9} in particular $f$ is a bijection, the desired result then follows from Proposition \ref{proposition_3_3} and Proposition \ref{proposition_3_4}.
\end{proof}

\vspace{25pt}

University of Warwick, Coventry, CV4 7AL, UK

\href{hanwen.liu@warwick.ac.uk}{hanwen.liu@warwick.ac.uk}

\end{document}